\documentclass[a4paper]{amsart}
\usepackage[utf8]{inputenc}
\usepackage[ngerman, english]{babel}
\usepackage{amsmath,amsthm,amssymb,amsfonts}
\usepackage{mathrsfs}
\usepackage{ifthen}
\usepackage{enumerate}
\usepackage{comment}
\usepackage[pdfauthor={Martin Halla},pdftitle={weak T-coercivity and T-compatibility},
unicode,colorlinks=true,linkcolor=black,citecolor=black,urlcolor=black,pagebackref=false,breaklinks]{hyperref}
\usepackage{hyphsubst}
\usepackage{xcolor}
\usepackage{calc}
\usepackage{graphicx}
\usepackage[toc,page]{appendix}

\makeatletter
\def\@hspace#1{\begingroup\setlength\dimen@{#1}\hskip\dimen@\endgroup}
\makeatother

\newtheorem{theorem}{Theorem}[section]
\newtheorem{lemma}[theorem]{Lemma}
\newtheorem{definition}[theorem]{Definition}

\newtheorem{corollary}[theorem]{Corollary}

\newtheorem{proposition}[theorem]{Proposition}

\newcommand{\spml}{\lambda}
\newcommand{\spl}{\langle}
\newcommand{\spr}{\rangle}
\newcommand{\fu}{u}
\newcommand{\fv}{v}

\newcommand{\ol}[1]{\overline{#1}}

\newcommand{\bpm}{\begin{pmatrix}}
\newcommand{\epm}{\end{pmatrix}}

\renewcommand{\dim}{\operatorname{dim}}

\newcommand{\setC}{\mathbb{C}}

\newcommand{\setN}{\mathbb{N}}

\definecolor{brickred}{rgb}{0.8, 0.25, 0.33}
\definecolor{bostonuniversityred}{rgb}{0.8, 0.0, 0.0}
\definecolor{cornellred}{rgb}{0.7, 0.11, 0.11}
\definecolor{corn}{rgb}{0.98, 0.93, 0.36}
\definecolor{schoolbusyellow}{rgb}{1.0, 0.85, 0.0}
\definecolor{TUblue}{rgb}{0,102,153}
\colorlet{TUbluelight}{TUblue!30!white}

\title[weak T-coercivity and T-compatibility]{Galerkin approximation of holomorphic eigenvalue problems:\\
weak T-coercivity and T-compatibility}
\author{Martin Halla}
\email{martin.halla@protonmail.com}
\subjclass[2010]{47J10, 65H17, 65N25.}
\date{August 12th, 2019.}
\keywords{holomorphic eigenvalue problem, non-linear eigenvalue problem, approximation analysis, T-coercivity.}

\begin{document}

\begin{abstract}
We consider Galerkin approximations of holomorphic Fredholm operator eigenvalue problems for which the operator values
don't have the structure ``coercive+compact''. In this case the regularity (in sense of
[O.\ Karma, Numer.\ Funct.\ Anal.\ Optim.\ 17 (1996)]) of Galerkin approximations is not unconditionally satisfied and
the question of convergence is delicate.
We report a technique to prove regularity of approximations which is applicable to a wide range of eigenvalue problems.
In particular, we introduce the concepts of weak T-coercivity and T-compatibility and prove that for weakly T-coercive
operators, T-compatibility of Galerkin approximations implies their regularity. 

Our framework immediately improves the results of
[T.\ Hohage, L.\ Nannen, BIT 55(1) (2015)], is immediately applicable to
analyze approximations of eigenvalue problems related to
[A.-S.\ Bonnet-Ben Dhia, C.\ Carvalho, P.\ Ciarlet, Num.\ Math.\ 138(4) (2018)]
and is already applied in
[G.\ Unger, preprint (2017)].
\end{abstract}
\maketitle

The analysis of approximations for holomorphic Fredholm operator
eigenvalue problems has a long history~\cite{GrigorieffJeggle:73}, \cite{VainikkoKarma:74}, \cite{Vainikko:76},
\cite{JeggleWendland:77}, \cite{Karma:96a}, \cite{Karma:96b} and is usually
performed in the framework of discrete approximation schemes~\cite{Stummel:70}
and regular approximations of operator functions~\cite{Grigorieff:73},
\cite{AnseloneTreuden:85}. In this framework a complete convergence analysis and
asymptotic error estimates for eigenvalues are given by Karma
in~\cite{Karma:96a}, \cite{Karma:96b}. If the discrete approximation scheme is
chosen as a Galerkin scheme, then the assumptions of~\cite{Karma:96a},
\cite{Karma:96b} reduce to a single non-trivial assumption: the regular approximation
property (see Definition~\ref{def:RegularApproximation} for the meaning of regularity).
If the operator values are of the form ``coercive+compact'', the regularity of
Galerkin approximations is unconditionally satisfied. However, if the operator
values are not of this kind the question of spectrally converging approximations
is very delicate. This can already be observed for linear eigenvalue problems,
see e.g.\ \cite{BoffiBrezziGastaldi:00}, \cite{ArnoldFalkWinther:10}. Though
it is little known how to prove regularity of approximations.
In Theorem~\ref{thm:Regularity} we report a new condition on the
Galerkin spaces to ensure the regularity of Galerkin approximations such
that~\cite{Karma:96a}, \cite{Karma:96b} can be applied. This condition is
stronger than the classical regularity condition. However, it suffices for a
wide variety of applications.
On the side, we report in Lemma~\ref{lem:Eigenspaces} new asymptotic error estimates
on eigenspaces for regular Galerkin approximations (which are not provided
by~\cite{Karma:96a}, \cite{Karma:96b}). The latter is an improvement
of Unger~\cite[Theorem~4.3.7]{Unger:09}.
We combine our approach with the results of~\cite{Karma:96a}, \cite{Karma:96b} in
Proposition~\ref{prop:SpectralConvergence} and Corollary~\ref{cor:SpectralConvergence}.\\

As preparation for the forthcoming concept of weakly T-coercive operators (operator functions) we remind the reader
how  Fredholmness of operators is usually established. In the case of coercive operators Fredholmness is trivial. The
same holds for weakly coercive operators $A$, i.e.\ $A$ is a compact perturbation of a coercive operator.
Else wise we may construct an isomorphism $T$ such that $T^*A$ is weakly coercive ($T^*$ denotes the adjoint operator
of $T$), which yields the Fredholmness of $A$. The name ``T-coercivity'' originates from Bonnet-Ben Dhia, Ciarlet,
Zw\"olf~\cite{BonnetBDCiarletZwoelf:10}. 
The notion was introduced to analyze differential operators with sign-changing coefficients in the principal part which
occur e.g.\ in the modeling of meta materials. The technique is also applied in the analysis of interior transmission
eigenvalue problems, see e.g.\ \cite{Chesnel:12a}, \cite{Chesnel:12b}.
Although as far as we know, the concept goes back to a remark by Buffa~\cite{Buffa:05} (wherein $T=\theta$) on 
non-coercive operators with applications to Maxwell equations.
For an operator $A$ to be (weakly) $T$-coercive means that $T^*A$ is already (weakly) coercive. However, in eigenvalue
problems the operator values will be in general not bijective (precisely at the eigenvalues). Thus the nomenclature of
T-coercivity is not meaningful for our purposes and we will rely on the term \emph{weak} T-coercivity. In general the
Galerkin spaces will not be $T$-invariant and hence one cannot reproduce the above analysis on the approximation level.
An invariance condition is indeed not necessary, but can be relaxed. We will make precise in which sense the Galerkin
spaces have to interact with the operator $T$ to ensure regularity. It will turn out that the existence of bounded
linear operators $T_n$ from the Galerkin spaces $X_n$ to themselves such that
\begin{align}
\lim_{n\in\setN} \|T-T_n\|_n=0,
\end{align}
with
\begin{align}\label{eq:DiscreteNorm}
\|T-T_n\|_n:=\sup_{u_n\in X_n\setminus\{0\}} \frac{\|T\fu_n\|_X}{\|\fu_n\|_X}
\end{align}
is sufficient. We call this property ``$T$-compatibility''.
The norm~\eqref{eq:DiscreteNorm} was termed ``discrete norm'' by Descloux, Nassif and Rappaz
\cite{DesclouxNassifRappaz:78a}, \cite{DesclouxNassifRappaz:78b} wherein it was used in a different but familiar context.
In our context it was already employed by Hohage and Nannen~\cite{HohageNannen:15} for the analysis of perfectly matched
layer and Hardy space infinite element methods in cylindrical waveguides; and also by Bonnet-Ben Dhia, Ciarlet and
Carvalho \cite{Carvalho:15}, \cite{BonnetBDCarvalhoCiarlet:18} for the analysis of finite element methods for equations
which involve meta materials. Both works \cite{HohageNannen:15}, \cite{BonnetBDCarvalhoCiarlet:18} prove weak T-coercivity
and T-compatibility. Thus our results can directly be applied to improve the results of \cite{HohageNannen:15} and
to establish convergence results for approximations of the eigenvalue problems related to \cite{BonnetBDCarvalhoCiarlet:18}.
Note that the negative material parameters in meta materials are e.g.\ of the kind $(1-1/\omega^{-2})^{-1}$ with
$\omega^2$ being the eigenvalue parameter. Hence such eigenvalue problems are indeed non-linear.

However, the original motivation for this article was to provide a framework for the convergence analysis of boundary
element discretizations of boundary integral formulations of Maxwell eigenvalue problems and is already applied by
Unger~\cite{Unger:17}. Although the Maxwell eigenvalue problem is of linear nature, its formulation as boundary integral
equation becomes non-linear due to the dependency of the fundamental solution on the frequency.\\

The remainder of this article is structured as follows. In Section~\ref{sec:wTc} we introduce the notion of weak
T-coercivity and T-compatibility. In Theorem~\ref{thm:Regularity} we prove that T-compatibility implies regularity.
In Section~\ref{sec:HolomorphicEVP} we report in Lemma~\ref{lem:Eigenspaces} an approximation result on eigenspaces for
regular Galerkin approximations of holomorphic Fredholm operator eigenvalue problems. We merge our results with the
results of Karma \cite{Karma:96a}, \cite{Karma:96b} in Proposition~\ref{prop:SpectralConvergence} and
Corollary~\ref{cor:SpectralConvergence}.

\section{Weak T-coercivity and T-compatibility}\label{sec:wTc}
Let $X$ be a Hilbert space with scalar body $\setC$ and scalar product $\spl\cdot,\cdot\spr_X$ and associated norm
$\|\cdot\|_X$. Let $L(X)$ be the space of bounded linear operators from $X$ to $X$ with operator norm
$\|A\|_{L(X)}:=\sup_{\fu\in X\setminus\{0\}}\|A\fu\|_X/\|\fu\|_X$ for $A\in L(X)$.
For $A\in L(X)$ we denote its adjoint operator by $A^*\in L(X)$, i.e.\
$\spl\fu,A^*\fv\spr_X = \spl A\fu,\fv \spr_X$ for all $\fu,\fv\in X$.
For a closed subspace $X_n\subset X$ let $L(X_n)$ be the space of bounded linear operators from
$X_n$ to $X_n$ with norm $\|A_n\|_{L(X_n)}:=\sup_{u_n\in X_n\setminus\{0\}}\|A_nu_n\|_X/\|u_n\|_X$ for
$A_n\in L(X_n)$ and denote $P_n$ the orthogonal projection from $X$ to $X_n$.
Henceforth we assume that $(X_n)_{n\in\setN}$ is a sequence of closed subspaces of $X$ such that $P_n$ converges
point-wise to the identity, i.e.\ $\lim_{n\in\setN} \|u-P_nu\|_X=0$ for each $u\in X$.

\begin{definition}\label{def:TcoerciveOperator}
Let $A,T\in L(X)$ and $T$ be bijective. The operator $A$ is called
\begin{enumerate}
 \item coercive, if $\inf_{\fu\in X\setminus\{0\}}|\spl A\fu,\fu \spr_X|/\|\fu\|_X^2>0$,
 \item weakly coercive, if there exists a compact operator $K\in L(X)$ such that $A+K$ is coercive,
 \item $T$-coercive if $T^*A$ is coercive,
 \item weakly $T$-coercive if $T^*A$ is weakly coercive.
\end{enumerate}
\end{definition}
Due to the Lemma of Lax-Milgram every coercive operator is invertible.
Every weakly $T$-coercive operator is Fredholm with index zero.
For a (weakly) coercive operator $A$ it is true that the Galerkin approximations $A_n=P_nA|_{X_n}\in L(X_n)$ inherit
the (weak) coercivity, while for (weakly) $T$-coercive operators it is in general wrong.

We note that if $T^*A$ is weakly coercive, then $AT^{-1}$ is so too. Vice-versa, if $AT$ is weakly coercive,
then so is $T^{-*}A$. Hence we could alternatively define $A$ to be (weakly) \emph{right} $T$-coercive, if
$AT$ is (weakly) coercive. However, we stick to the former variant because it is more convenient.

For an operator $T\in L(X)$ or $T\in L(X_n),$ or a sum of such we define the ``discrete norm''
\begin{align}
\|T\|_n:=\sup_{\fu_n\in X_n\setminus\{0\}}\frac{\|T\fu_n\|_X}{\|\fu_n\|_X}
=\|T\|_{L(X_n,X)}=\|TP_n\|_{L(X)}.
\end{align}

\begin{definition}\label{def:ConvergenceInDiscreteNorm}
Consider $T\in L(X)$ and $(T_n\in L(X_n))_{n\in\setN}$. We say that $T_n$ converges to $T$ in discrete norm, if
\begin{align}
\lim_{n\in\setN} \|T-T_n\|_n=0.
\end{align}
\end{definition}

We define in the following what we mean by $T$-compatible approximations of weakly $T$-coercive operators.
\begin{definition}\label{def:TCompatibleApproximation}
Let $A\in L(X)$ be weakly $T$-coercive. Then we call the sequence of Galerkin approximations
$(A_n:=P_nA|_{X_n}\in L(X_n))_{n\in\setN}$ $T$-compatible, if $(A_n)_{n\in\setN}$ is a sequence of index zero
Fredholm operators and there exists a sequence of index zero Fredholm operators $(T_n\in L(X_n))_{n\in\setN}$ such
that $T_n$ converges to $T$ in discrete norm: $\lim_{n\in\setN}\|T-T_n\|_n=0$.
\end{definition}

\begin{definition}\label{def:CompactSequence}
A sequence $(\fu_n \in X)_{n\in\setN}$ is said to be compact, if for every subsequence exists in turn a converging
subsubsequence.
\end{definition}

\begin{definition}\label{def:RegularApproximation}
A sequence $(A_n\in L(X_n))_{n\in\setN}$ is called regular, if for every bounded sequence $(\fu_n\in X_n)_{n\in\setN}$
the compactness of $(A_n\fu_n)_{n\in\setN}$ already implies the compactness of $(\fu_n)_{n\in\setN}$.
\end{definition}
Next we briefly elaborate on the notion of regularity for readers who are totally unfamiliar with this concept.
Regularity of Galerkin approximations is a meaningful generalization of stability and well suited for the approximation
analysis of eigenvalue problems. Consider for example bijective $A\in L(X)$ and its Galerkin approximation
$(A_n:=P_nA|_{X_n}\in L(X_n))_{n\in\setN}$. In this case regularity of $(A_n)_{n\in\setN}$ implies stability:
Assume that $(A_n)_{n\in\setN}$ is not stable. Thus there exists $(\fu_n\in X_n)_{n\in\setN}$ with $\|\fu_n\|_X=1$ for
each $n\in\setN$ such that $\lim_{n\in\setN}\|A_n\fu_n\|_X=0$. If $(A_n)_{n\in\setN}$ is regular, there exists a
subsequence $n(m)_{m\in\setN}$ and $\fu\in X$ such that $\lim_{m\in\setN}\fu_{n(m)}=\fu$. It follows $A\fu=\lim_{m\in\setN}
A_{n(m)}\fu_{n(m)}=0$. Since $A$ is bijective, it follows $\fu=0$ which is a contradiction to $\|\fu_{n(m)}\|_X=1$.

On the other hand, consider a holomorphic Fredholm operator function $A(\cdot)\colon\Lambda\subset\setC\to L(X)$
with non-empty resolvent set and sequences $(\lambda_n\in\Lambda,\fu_n\in X_n)_{n\in\setN}$ of eigenvalues with
normalized eigenelements of the Galerkin approximation (i.e.\ $A_n(\lambda_n)\fu_n=0$) such that
$\lim_{n\in\setN}\lambda_n=\lambda\in\Lambda$ (see Section~\ref{sec:HolomorphicEVP} for definitions and details).
If $A_n(\lambda)$ is regular for each $\lambda_n\in\Lambda$, then $\lambda$ is indeed an eigenvalue of $A(\cdot)$ (i.e.\
there occurs no spectral pollution):
Due to the continuity of $A_n(\cdot)$ with respect to $\lambda$, $A_n(\lambda_n)\fu_n=0$ implies
$\lim_{n\in\setN} A_n(\lambda)\fu_n=0$. If $(A_n(\lambda))_{n\in\setN}$ is regular, there exists a
subsequence $n(m)_{m\in\setN}$ and $\fu\in X$ such that $\lim_{m\in\setN}\fu_{n(m)}=\fu$. It follows
$A(\lambda)\fu=\lim_{m\in\setN} A_{n(m)}(\lambda)\fu_{n(m)}=0$ and $\|\fu\|_X=\lim_{m\in\setN} \|\fu_{n(m)}\|_X=1$, i.e.\
$\lambda$ is an eigenvalue of $A(\cdot)$ with normalized eigenelement $u$.\\

Our next goal is to prove in Theorem~\ref{thm:Regularity} that $T$-compatible Galerkin approximations of weakly
$T$-coercive operators are regular. In preparation we formulate the next two lemmata.
\begin{lemma}\label{lem:NormEstimates}
Let $T\in L(X)\setminus\{0\}$ and $(T_n\in L(X_n))_{n\in\setN}$ be a sequence of operators with $T_n\in L(X_n)$ and
$\lim_{n\in\setN} \|T-T_n\|_n=0$. Then there exist a constant $c>0$ and an index $n_0\in\setN$ such that
\begin{align}
\|T_n\|_{L(X_n)}, \|T_n\|_{L(X_n)}^{-1}\leq c
\end{align}
for all $n>n_0$. If $T$ is bijective and $T_n$ is Fredholm with index zero for each $n\in\setN$, then there
exist a constant $c>0$ and an index $n_0\in\setN$ such that $T_n$ is also bijective for all $n>n_0$ and
\begin{align}
\|(T_n)^{-1}\|_{L(X_n)} \leq c.
\end{align}
\end{lemma}
\begin{proof}
Let $\fu_n\in X_n$. With the triangle inequality we deduce
\begin{align*}
\|T_n\fu_n\|_X \leq \|T\fu_n\|_X + \|(T-T_n)\fu_n\|_X
\end{align*}
and hence
\begin{align*}
\|T_n\|_{L(X_n)}&\leq\|T\|_{L(X)}+\|T-T_n\|_n.
\end{align*}
Since $\lim_{n\in\setN}\|T-T_n\|_n=0$ the right hand side of the previous inequality is bounded. Similar, with the
inverse triangle inequality we deduce
\begin{align*}
\|T_n\fu_n\|_X \geq \|T\fu_n\|_X -\|(T-T_n)\fu_n\|_X
\end{align*}
and hence
\begin{align*}
\|T_n\|_{L(X_n)}&\geq\|T\|_n-\|T-T_n\|_n.
\end{align*}
It hold $\lim_{n\in\setN}\|T\|_n=\|T\|_{L(X)}>0$ and $\lim_{n\in\setN}\|T-T_n\|_n=0$. Thus let $n_0>0$ be such that
$|\|T\|_n-\|T\|_{L(X)}|<\|T\|_{L(X)}/3$ and $\|T-T_n\|_n<\|T\|_{L(X)}/3$ for all $n>n_0$. It follows
\begin{align*}
\|T_n\|_{L(X_n)}&\geq \|T\|_{L(X)}/3>0
\end{align*}
for all $n>n_0$.
For the last claim let $n_0>0$ be such that $\|T-T_n\|_n<\frac{1}{2\|T^{-1}\|_{L(X)}}$ for all $n>n_0$.
Again with the inverse triangle inequality and
\begin{align*}
\inf_{\fu\in X, \|\fu\|_X=1} \|T\fu\|_X=1/\|T^{-1}\|_{L(X)}>0
\end{align*}
it follows
\begin{align*}
\inf_{\fu_n\in X_n, \|\fu_n\|_X=1} \|T_n\fu_n\|_X
&\geq \inf_{\fu\in X, \|\fu\|_X=1} \|T\fu\|_X-\|T-T_n\|_n\\
&\geq 1/(2\|T^{-1}\|_{L(X)})
\end{align*}
for all $n>n_0$. We deduce that $T_n$ is injective. Since $T_n$ is Fredholm with index zero its bijectivity follows.
The norm estimate holds due to $\inf_{\fu_n\in X_n, \|\fu_n\|_X=1} \|T_n\fu_n\|_X=1/\|T_n^{-1}\|_{L(X_n)}$.
\end{proof}

\begin{lemma}\label{lem:Stability}
Let $A\in L(X)$ be weakly $T$-coercive and $K\in L(X)$ be compact such that $T^*A+K$ is coercive.
Let $(A_n:=P_nA|_{X_n}\in L(X_n))_{n\in\setN}$ be a $T$-compatible Galerkin approximation of $A$.
Then there exist $n_0\in\setN$ and $c>0$, such that $A_n+P_nT^{-*}K|_{X_n}\in L(X_n)$ is invertible and
\begin{align}
\|\big(A_n+P_nT^{-*}K|_{X_n}\big)^{-1}\|_{L(X_n)}\leq c
\end{align}
for all $n>n_0$.
\end{lemma}
\begin{proof}
Let $n$ be large enough such that $T_n$ is bijective (see Lemma~\ref{lem:NormEstimates}). We compute
\begin{align*}
\inf_{\fu_n\in X_n\setminus\{0\}}\sup_{\fv_n\in X_n\setminus\{0\}}
  &\frac{|\langle (A+T^{-*}K)\fu_n,\fv_n \rangle_X|}{\|\fu_n\|_X\|\fv_n\|_X}\\
&\geq \inf_{\fu_n\in X_n\setminus\{0\}}\sup_{\fv_n\in X_n\setminus\{0\}}
\frac{|\langle (A+T^{-*}K)\fu_n,T_n\fv_n \rangle_X|}{\|T_n\|_{L(X_n)}\|\fu_n\|_X\|\fv_n\|_X}\\
&\geq \inf_{\fu_n\in X_n\setminus\{0\}}\sup_{\fv_n\in X_n\setminus\{0\}}
\frac{|\langle ((A+T^{-*}K)\fu_n,T\fv_n \rangle_X|}{\|T_n\|_{L(X_n)}\|\fu_n\|_X\|\fv_n\|_X}\\
& \qquad-\frac{\|A+T^{-*}K\|_{L(X)}}{\|T_n\|_{L(X_n)}}\|T-T_n\|_n\\
&= \inf_{\fu_n\in X_n\setminus\{0\}}\sup_{\fv_n\in X_n\setminus\{0\}}
\frac{|\langle T^*(A+T^{-*}K)\fu_n,\fv_n \rangle_X|}{\|T_n\|_{L(X_n)}\|\fu_n\|_X\|\fv_n\|_X}\\
& \qquad-\frac{\|A+T^{-*}K\|_{L(X)}}{\|T_n\|_{L(X_n)}}\|T-T_n\|_n\\
&= \inf_{\fu_n\in X_n\setminus\{0\}}\sup_{\fv_n\in X_n\setminus\{0\}}
\frac{|\langle (T^*A+K)\fu_n,\fv_n \rangle_X|}{\|T_n\|_{L(X_n)}\|\fu_n\|_X\|\fv_n\|_X}\\
& \qquad-\frac{\|A+T^{-*}K\|_{L(X)}}{\|T_n\|_{L(X_n)}}\|T-T_n\|_n\\
&\geq \tilde c\|T_n\|_{L(X_n)}^{-1}-\frac{\|A+T^{-*}K\|_{L(X)}}{\|T_n\|_{L(X_n)}}\|T-T_n\|_n
\end{align*}
with coercivity constant
\begin{align*}
\tilde c:=\inf_{\fu\in X\setminus\{0\}} |\langle (T^*A+K)\fu,\fu \rangle_X| / \|\fu\|_X^2>0.
\end{align*}
Since  $\|T_n\|_{L(X_n)}$ is uniformly bounded from above and below (see Lemma~\ref{lem:NormEstimates}) and
$T_n$ converges to $T$ in discrete norm by assumption, it follows the existence of $n_0\in\setN$ and $c>0$ such that
\begin{align*}
\inf_{\fu_n\in X_n\setminus\{0\}}\sup_{\fv_n\in X_n\setminus\{0\}}
\frac{|\langle (A+T^{-*}K)\fu_n,\fv_n \rangle_X|}{\|\fu_n\|_X\|\fv_n\|_X} \geq c
\end{align*}
for all $n>n_0$. Hence $A_n+P_nT^{-*}K|_{X_n}$ is injective. Since $A_n$ is Fredholm with index zero and
$K$ is compact, $A_n+P_nT^{-*}K|_{X_n}$ is Fredholm with index zero too. Thus $A_n+P_nT^{-*}K|_{X_n}$ is bijective.
The norm estimate follows now from
\begin{align*}
\inf_{\fu_n\in X_n\setminus\{0\}}\sup_{\fv_n\in X_n\setminus\{0\}} \frac{|\spl B_n\fu_n,\fv_n\spr_X|}{\|\fu_n\|_X\|\fv_n\|_X}
&=\inf_{\fu_n\in X_n\setminus\{0\}} \frac{\|B_n\fu_n\|_X}{\|\fu_n\|_X}\\
&=\left(\sup_{\fu_n\in X_n\setminus\{0\}} \frac{\|\fu_n\|_X}{\|B_n\fu_n\|_X}\right)^{-1}\\
&=\|B_n^{-1}\|_{L(X_n)}^{-1}
\end{align*}
for any bijective $B_n\in L(X_n)$.
\end{proof}

\begin{theorem}\label{thm:Regularity}
Let $A\in L(X)$ be weakly $T$-coercive and
\begin{align*}
(A_n:=P_nA|_{X_n}\in L(X_n))_{n\in\setN}
\end{align*}
be a $T$-compatible Galerkin approximation. Then $(A_n)_{n\in\setN}$ is regular.
\end{theorem}
\begin{proof}
Without loss of generality let $(\fu_n\in L(X_n))_{n\in\setN}$ be a bounded sequence,
$(A_n\fu_n)_{n\in\setN}$ and $\fu'\in X$ be such that $\lim_{n\in\setN}A_n\fu_n=\fu'$.
Let $K\in L(X)$ be compact such that $T^*A+K$ is coercive. Let $\tilde A:=A+T^{-*}K$ and $\tilde A_n:=P_n\tilde A|_{X_n}$.
Since $K$ is compact and $(\fu_n)_{n\in\setN}$ is bounded, there exist a subsequence
$(\fu_{n(m)})_{m\in\setN}$ and $\fu''\in X$ such that $\lim_{m\in\setN}T^{-*}Ku_{n(m)}=\fu''$. It follows
\begin{align*}
\lim_{m\in\setN} \tilde A_{n(m)} \fu_{n(m)}=\fu'+\fu''.
\end{align*}
Due to Lemma~\ref{lem:Stability} there exist $c>0$ and $m_0\in\setN$, such that for all $m>m_0$ operator
$\tilde A_{n(m)}$ is invertible and $\|\tilde A_{n(m)}^{-1}\|_{L(X_{n(m)})}\leq c$.
For $m>m_0$ we compute
\begin{align*}
\|\fu_{n(m)}-&\tilde A^{-1}(u'+u'')\|_X\\
&\leq \|\fu_{n(m)}-P_{n(m)}\tilde A^{-1}(u'+u'')\|_X + \|(I-P_{n(m)})\tilde A^{-1}(u'+u'')\|_X\\
&\leq c \|\tilde A_{n(m)}\fu_{n(m)}-\tilde A_{n(m)}P_{n(m)}\tilde A^{-1}(u'+u'')\|_X\\
&+\|(I-P_{n(m)})\tilde A^{-1}(u'+u'')\|_X\\
&\leq c\|\tilde A_{n(m)}\fu_{n(m)}-(\fu'+\fu'')\|_X \\
&+c\|(I-\tilde A_{n(m)}P_{n(m)}\tilde A^{-1})(u'+u'')\|_X\\
&+\|(I-P_{n(m)})\tilde A^{-1}(u'+u'')\|_X.
\end{align*}
The first term on the right hand side of the latter inequality converges to zero, as previously discussed.
The second and third term converge to zero, because $(P_{n(m)})_{m\in\setN}$ converges point-wise to the identity.
Hence
\begin{align*}
\lim_{m\in\setN} \fu_{n(m)} = \tilde A^{-1}(u'+u'').
\end{align*}
\end{proof}

\section{Holomorphic eigenvalue problems}\label{sec:HolomorphicEVP}
We refer the reader to \cite{GohbergLeiterer:09} and \cite[Appendix]{KozlovMazya:99}
for theory on holomorphic (Fredholm) operator functions.
Let $\Lambda\subset\setC$ be an open, connected and non-empty subset of $\setC$.
Let $A(\cdot)\colon\Lambda\to L(X)$ be an operator function.
An operator function $A(\cdot)$ is called holomorphic, if it is complex differentiable.
An operator function $A(\cdot)$ is called Fredholm, if $A(\spml)$ is Fredholm for each $\spml\in\Lambda$.
We denote the resolvent set and spectrum of an operator function $A(\cdot)\colon\Lambda\to L(X)$ as
\begin{align}
\rho\big(A(\cdot)\big):=\{\spml\in\Lambda\colon A(\spml)\text{ is invertible}\} \qquad\text{and}\qquad
\sigma\big(A(\cdot)\big):=\Lambda\setminus\rho\big(A(\cdot)\big).
\end{align}
For an operator function $A(\cdot)\colon\Lambda\to L(X)$ we denote by $A^*(\cdot)$ the operator function defined by
$A^*(\spml):=A(\spml)^*$ for each $\spml\in\Lambda$ and by $A^{-1}(\cdot)\colon\rho\big(A(\cdot)\big)\to L(X)$ the
operator function defined by $A^{-1}(\spml):=A(\spml)^{-1}$ for each $\spml\in\rho\big(A(\cdot)\big)$. Note that for
a holomorphic operator function $A(\cdot)\colon\Lambda\to L(X)$ the operator function defined by 
$\spml\mapsto A^*(\ol{\spml})$ is holomorphic as well. For a holomorphic operator function $A(\cdot)\colon\Lambda\to L(X)$
denote by $A^{(j)}(\cdot)\colon\Lambda\to L(X)$ the $j^{th}$ derivative of $A(\cdot)\colon\Lambda\to L(X)$.
It is well known (see e.g.\ \cite[Theorem~8.2]{GohbergGoldbergKaashoek:90}) that for a holomorphic Fredholm operator
function $A(\cdot)\colon\Lambda\to L(X)$ such that $A(\spml)$ is bijective for at least one $\spml\in\Lambda$, the
spectrum $\sigma\big(A(\cdot)\big)$ is discrete, has no accumulation points in $\Lambda$ and every
$\spml\in\sigma\big(A(\cdot)\big)$ is an eigenvalue. That is, there exists $\fu\in X$ such that $A(\spml)\fu=0$. In this
case we call $\fu$ an eigenelement. An ordered collection of elements $(\fu_0,\fu_1,\dots,\fu_{m-1})$ in $X$ is called a
Jordan chain at $\spml$ if $\fu_0$ is an eigenelement corresponding to $\spml$ and if
\begin{align}
\sum_{j=0}^l\frac{1}{j!}A^{(j)}(\spml)\fu_{l-j} \quad\text{for }l=0,1,\dots,m-1.
\end{align}
The elements of a Jordan chain are called generalized eigenelements and the closed linear hull of all generalized
eigenelements of $A(\cdot)$ at $\spml$ is called the generalized eigenspace $G(A(\cdot),\spml)$ for $A(\cdot)$ at $\spml$.
For an eigenelement $\fu\in\ker A(\spml)\setminus\{0\}$ we denote by $\varkappa(A(\cdot),\spml,u)$ the maximal length of
a Jordan chain at $\spml$ beginning with $\fu$ and
\begin{align}
\varkappa(A(\cdot),\spml):=\max_{\fu\in\ker A(\spml)\setminus\{0\}} \varkappa(A(\cdot),\spml,\fu).
\end{align}
The maximal length of a Jordan chain $\varkappa(A(\cdot),\spml)$ is always finite,
see e.g.~\cite[Lemma~A.8.3]{KozlovMazya:99}.
Next we generalize Definitions~\ref{def:TcoerciveOperator},~\ref{def:TCompatibleApproximation},
\ref{def:RegularApproximation} and Theorem~\ref{thm:Regularity} to operator functions.

\begin{definition}\label{def:TcoerciveOperatorFunction}
Let $A(\cdot), T(\cdot)\colon\Lambda\to L(X)$ be operator functions and $\rho\big(T(\cdot)\big)=\Lambda$.
$A(\cdot)$ is (weakly) ($T(\cdot)$-)coercive, if $A(\lambda)$ is (weakly) ($T(\lambda)$-)coercive for each
$\lambda\in\Lambda$.
\end{definition}

\begin{definition}\label{def:TCompatibleApproximationOperatorFunction}
Let $A(\cdot)\colon\Lambda\to L(X)$ be weakly $T(\cdot)$-coercive. Then we call the sequence of Galerkin approximations
$(A_n(\cdot):=P_nA(\cdot)|_{X_n}\colon \Lambda\to L(X_n))_{n\in\setN}$ $T(\cdot)$-compatible, if
$(A_n(\lambda))_{n\in\setN}$ is $T(\lambda)$ compatible for each $\lambda\in\Lambda$.
\end{definition}

\begin{definition}\label{def:RegularApproximationOperatorFunction}
Let $A(\cdot)\colon\Lambda\to L(X)$ be an operator function. The sequence of Galerkin approximations
$(A_n(\cdot):=P_nA(\cdot)|_{X_n}\colon \Lambda\to L(X_n))_{n\in\setN}$ is regular, if
$(A_n(\lambda))_{n\in\setN}$ is regular for each $\lambda\in\Lambda$
\end{definition}

\begin{theorem}\label{thm:RegularityOperatorFunction}
Let $A(\cdot)\colon\Lambda\to L(X)$ be weakly $T(\cdot)$-coercive and
\begin{align*}
(A_n(\cdot):=P_nA(\cdot)|_{X_n}\colon\Lambda\to L(X_n))_{n\in\setN}
\end{align*}
be a $T(\cdot)$-compatible Galerkin approximation. Then $(A_n(\cdot))_{n\in\setN}$ is regular.
\end{theorem}
\begin{proof}
Follows from Theorem~\ref{thm:RegularityOperatorFunction}.
\end{proof}

Next we prepare to apply~\cite{Karma:96a},~\cite{Karma:96b}.
\begin{lemma}\label{lem:DiscreteApproximationScheme}
Let $A(\cdot)\colon\Lambda\to L(X)$ be a holomorphic Fredholm operator function and let $(X_n)_{n\in\setN}$ be a
sequence of closed subspaces of $X$ with orthogonal projections $P_n$ onto $X_n$, such that $(P_n)_{n\in\setN}$
converges point-wise to the identity. Then the Galerkin scheme $\big(P_nA(\cdot)|_{X_n})_{n\in\setN}$ is a
discrete approximation scheme in the sense of~\cite{Karma:96a}.
\end{lemma}
\begin{proof}
For a Galerkin scheme it holds with the notation of~\cite{Karma:96a}
\begin{align*}
U&=V=X, \qquad X_n=Y_n=X_n, \qquad A_n(\cdot)=P_nA(\cdot)|_{X_n}, \qquad p_n=q_n=P_n.
\end{align*}
Assumptions a1)-a4) of~\cite{Karma:96a} follow all from the point-wise convergence of $P_n$.
\end{proof}

Next we generalize Theorem~4.3.7 of~\cite{Unger:09}.
\begin{lemma}\label{lem:Eigenspaces}
Let $\Lambda\subset\setC$ be open, $X$ be a Hilbert space and $L(X)$ be the space of bounded linear operators
from $X$ to $X$. Let $A(\cdot)\colon\Lambda\to L(X)$ be a holomorphic Fredholm operator function with non-empty
resolvent set and $(X_n)_{n\in\setN}$ be a sequence of closed subspaces of $X$ with orthogonal projections $P_n$ onto
$X_n$, such that $(P_n)_{n\in\setN}$ converges point-wise to the identity, i.e.\ $\lim_{n\in\setN}\|u-P_nu\|_X=0$
for all $u\in X$.
Let $A_n(\cdot)\colon\Lambda\to L(X_n)$ be the Galerkin approximation of $A(\cdot)$ defined by
$A_n(\spml):=P_nA(\spml)|_{X_n}$ for each $\spml\in\Lambda$.
Let the assumptions of~\cite[Theorem~2, Theorem~3]{Karma:96a} and \cite[Theorem~2, Theorem~3]{Karma:96b}
be satisfied. Let $\tilde\Lambda\subset\Lambda$ be a compact set with rectifiable boundary
$\partial\tilde\Lambda\subset\rho\big(A(\cdot)\big)$ and $\tilde\Lambda\cap\sigma\big(A(\cdot)\big)=\{\spml_0\}$.
Then there exist $n_0\in\setN$ and $c>0$ such that for all $n>n_0$
\begin{align}
  \inf_{\fu_0\in\ker A(\spml_0)} \|\fu_n-\fu_0\|_X \leq c \Big(|\spml_n-\spml_0|+
  \max_{\substack{u'_0\in\ker A(\spml_0)\\\|u'_0\|_X\leq1}} \inf_{u'_n\in X_n} \|u'_0-u'_n\|_X\Big)
\end{align}
for all $\spml_n\in\sigma\big(A_n(\cdot)\big)\cap\tilde\Lambda$ and all $\fu_n\in \ker A_n(\spml_n)$ with $\|\fu_n\|_X=1$.
\end{lemma}
\begin{proof}
We proceed as in~\cite{Unger:09}: Theorem~4.3.7 of~\cite{Unger:09} requires a special
form of the operator function $A(\cdot)$. However its proof uses this assumption only to
apply Lemma~4.2.1 of~\cite{Unger:09}. Hence we need to establish the result of~\cite[Lemma~4.2.1]{Unger:09} without the
assumption on the special form of $A(\cdot)$.
However, the result of~\cite[Lemma~4.2.1]{Unger:09} already follows from~\cite[Theorem~2 ii)]{Karma:96a}.
\end{proof}

Next we apply~\cite{Karma:96a},~\cite{Karma:96b} and Lemma~\ref{lem:Eigenspaces}.
\begin{proposition}\label{prop:SpectralConvergence}
Let $\Lambda\subset\setC$ be open, connected and non-empty, $X$ be a Hilbert space and $L(X)$ be the space of bounded
linear operators from $X$ to itself. Let $A(\cdot)\colon\Lambda\to L(X)$ be a holomorphic Fredholm operator function
with non-empty resolvent set $\rho\big(A(\cdot)\big)\neq\emptyset$.
Let $(X_n)_{n\in\setN}$ be a sequence of closed subspaces of $X$ with orthogonal projections $P_n$ onto $X_n$,
such that $(P_n)_{n\in\setN}$ converges point-wise to the identity, i.e.\ $\lim_{n\in\setN}\|u-P_nu\|_X=0$
for each $u\in X$. Let $A_n(\cdot)\colon\Lambda\to L(X_n)$ be the Galerkin approximation of $A(\cdot)$ defined by
$A_n(\spml):=P_nA(\spml)|_{X_n}$ for each $\spml\in\Lambda$. Assume that $A_n(\lambda)$ is Fredholm with index zero
for each $\lambda\in\Lambda$ and $n\in\setN$. Assume that $(A_n(\cdot))_{n\in\setN}$ is a regular
approximation of $A(\cdot)$ (see Definition~\ref{def:RegularApproximationOperatorFunction}).
Then the following results hold.
\begin{enumerate}[i)]
 \item\label{item:SP-a} For every eigenvalue $\spml_0$ of $A(\cdot)$ exists a sequence $(\spml_n)_{n\in\setN}$
 converging to $\spml_0$ with $\spml_n$ being an eigenvalue of $A_n(\cdot)$ for almost all $n\in\setN$.
 \item\label{item:SP-b} Let $(\spml_n, u_n)_{n\in\setN}$ be a sequence of normalized eigenpairs of $A_n(\cdot)$, i.e.\
 \begin{align*}
 A_n(\spml_n)u_n=0,
 \end{align*}
 and $\|u_n\|_X=1$, so that $\spml_n\to \spml_0\in\Lambda$, then
 \begin{enumerate}[a)]
  \item $\spml_0$ is an eigenvalue of $A(\cdot)$,
  \item $(u_n)_{n\in\setN}$ is a compact sequence and its cluster points are normalized eigenelements of $A(\spml_0)$.
 \end{enumerate}
 \item\label{item:SP-c} For every compact $\tilde\Lambda\subset\rho(A)$ the sequence $(A_n(\cdot))_{n\in\setN}$ is
 stable on  $\tilde\Lambda$, i.e.\ there exist $n_0\in\setN$ and $c>0$ such that $\|A_n(\spml)^{-1}\|_{L(X_n)}\leq c$ for
 all  $n>n_0$ and all $\spml\in\tilde\Lambda$.
 \item\label{item:SP-d} \label{item:Stability}For every compact $\tilde\Lambda\subset\Lambda$ with rectifiable boundary
 $\partial\tilde\Lambda\subset\rho\big(A(\cdot)\big)$ exists an index $n_0\in\setN$ such that
 \begin{align}
 \dim G(A(\cdot),\spml_0) = \sum_{\spml_n\in\sigma\left(A_n(\cdot)\right)\cap\tilde\Lambda}
 \dim G(A_n(\cdot),\spml_n).
 \end{align}
 for all $n>n_0$, whereby $G(B(\cdot),\spml)$ denotes the generalized eigenspace of an operator function
 $B(\cdot)$ at $\spml\in\Lambda$.
\end{enumerate}
Let $\tilde\Lambda\subset\Lambda$ be a compact set with rectifiable boundary
$\partial\tilde\Lambda\subset\rho\big(A(\cdot)\big)$, $\tilde\Lambda\cap\sigma\big(A(\cdot)\big)=\{\spml_0\}$ and
\begin{align}\label{eq:deltan}
\begin{split}
\delta_n&:=\max_{\substack{u_0\in G(A(\cdot),\spml_0)\\\|u_0\|_X\leq1}} \, \inf_{u_n\in X_n} \|u_0-u_n\|_X,\\
\delta_n^*&:=\max_{\substack{u_0\in G(A^*(\ol{\cdot}),\spml_0)\\\|u_0\|_X\leq1}} \, \inf_{u_n\in X_n} \|u_0-u_n\|_X,
\end{split}
\end{align}
whereby $\overline{\spml_0}$ denotes the complex conjugate of $\spml_0$ and $A^*(\cdot)$
the adjoint operator function of $A(\cdot)$ defined by $A^*(\spml):=A(\spml)^*$ for each $\spml\in\Lambda$.
Then there exist $n\in\setN$ and $c>0$ such that for all $n>n_0$
\begin{enumerate}[i)]
\setcounter{enumi}{4}
  \item\label{item:SP-e}
  \begin{align}
  |\spml_0-\spml_n|\leq c(\delta_n\delta_n^*)^{1/\varkappa\left(A(\cdot),\spml_0\right)}
  \end{align}
  for all $\spml_n\in\sigma\big(A_n(\cdot)\big)\cap\tilde\Lambda$, whereby
  $\varkappa\left(A(\cdot),\spml_0\right)$ denotes the maximal length of a Jordan
  chain of $A(\cdot)$ at the eigenvalue $\spml_0$,
  \item\label{item:SP-f}
  \begin{align}
  |\spml_0-\bar \spml_n|\leq c\delta_n\delta_n^*
  \end{align}
  whereby $\bar \spml_n$ is the weighted mean of all the eigenvalues of $A_n(\cdot)$ in $\tilde\Lambda$
  \begin{align}
  \bar \spml_n:=\sum_{\spml\in\sigma\left(A_n(\cdot)\right)\cap\tilde\Lambda}\spml\,
  \frac{\dim G(A_n(\cdot),\spml)}{\dim G(A(\cdot),\spml_0)},
  \end{align}
  \item\label{item:SP-g}
  \begin{align}
  \begin{split}
  \inf_{u_0\in\ker A(\spml_0)} \|u_n-u_0\|_X &\leq c \Big(|\spml_n-\spml_0|+
  \max_{\substack{u'_0\in\ker A(\spml_0)\\\|u_0'\|_X\leq1}} \inf_{u'_n\in X_n} \|u'_0-u'_n\|_X\Big)\\
  &\leq c\big(c(\delta_n\delta_n^*)^{1/\varkappa\left(A(\cdot),\spml_0\right)} + \delta_n\big)
  \end{split}
  \end{align}
  for all $\spml_n\in\sigma\big(A_n(\cdot)\big)\cap\tilde\Lambda$ and all $u_n\in \ker A_n(\spml_n)$ with $\|u_n\|_X=1$.
\end{enumerate}
\end{proposition}
\begin{proof}
The first three claims follow with~\cite[Theorem~2]{Karma:96a}, if we can proof that the required assumptions are
satisfied. First of all a Galerkin scheme is a discrete approximation scheme due to
Lemma~\ref{lem:DiscreteApproximationScheme}. The operator function $A(\cdot)$ are holomorphic by
assumption. It follows that $A_n(\cdot):=P_nA(\cdot)P_n|_{X_n}$ is also holomorphic.
$A(\cdot)$ and $A_n(\cdot)$ are index zero Fredholm operator functions by assumption.
Assumption b1 $\rho\big(A(\cdot)\big)\neq\emptyset$ is also an assumption of this theorem.
Assumption b2 follows from Lemma~\ref{lem:Stability} (at least for sufficiently large $n$). Assumption b3 follows from
$\|A_n(\spml)\|_{L(X_n)}\leq\|A(\spml)\|_{L(X)}$. Assumption b4 follows from the point-wise
convergence of the projections $P_n$. Assumption b5 is also an assumption of this theorem.

The fourth claim follows with~\cite[Theorem~3]{Karma:96a}, if we can proof the required assumption (R).
We can chose $r_n$ as injection, i.e.\ $r_nx_n:=x_n$. Hence $\|r_n\|=1$. Since $p_n=P_n$ ii) follows from
the point-wise convergence of the projections $P_n$.

The fifth and sixth claim follow with~\cite[Theorem~2, Theorem~3]{Karma:96b}, if we can proof their required assumptions.
Assumption a1-a4 are canonical satisfied by Galerkin schemes. We already proved that Assumptions b1-b5
are satisfied. We can chose $p_n'=p_n=q'_n=q_n=P_n$. For \cite[Theorem~3]{Karma:96b} we can chose the same $r_n$
as before.

For the proof of the seventh claim we refer to Lemma~\ref{lem:Eigenspaces}.
\end{proof}

Finally we combine Theorem~\ref{thm:RegularityOperatorFunction} and Proposition~\ref{prop:SpectralConvergence}.
\begin{corollary}\label{cor:SpectralConvergence}
Let $\Lambda\subset\setC$ be open, connected and non-empty, $X$ be a Hilbert space and $L(X)$ be the space of bounded
linear operators from $X$ to $X$. Let $A(\cdot)\colon\Lambda\to L(X)$ be a holomorphic 
weakly $T(\cdot)$-coercive operator function (see Definition~\ref{def:TcoerciveOperatorFunction})
with non-empty resolvent set $\rho\big(A(\cdot)\big)\neq\emptyset$.
Let $(X_n)_{n\in\setN}$ be a sequence of closed subspaces of $X$ with orthogonal projections $P_n$ onto $X_n$,
such that $(P_n)_{n\in\setN}$ converges point-wise to the identity, i.e.\ $\lim_{n\in\setN}\|u-P_nu\|_X=0$
for each $u\in X$. Let $A_n(\cdot)\colon\Lambda\to L(X_n)$ be the Galerkin approximation of $A(\cdot)$ defined by
$A_n(\spml):=P_nA(\spml)|_{X_n}$ for each $\spml\in\Lambda$. Assume that $A_n(\cdot)$ is 
$T(\cdot)$-compatible (see Definition~\ref{def:TCompatibleApproximationOperatorFunction}).
Then results \ref{item:SP-a})-\ref{item:SP-g}) of Proposition~\ref{prop:SpectralConvergence} hold.
\end{corollary}
\begin{proof}
Since $A(\cdot)$ is weakly $T(\cdot)$-coercive, it is Fredholm with index zero.
Since $A_n(\cdot)$ is $T(\cdot)$-compatible, it is Fredholm with index zero and regular.
\end{proof}

\bibliographystyle{amsplain}
\bibliography{bibliography}
\end{document}